\documentclass[11pt]{article}

\usepackage{amsthm,amsmath,amsfonts,amssymb,graphicx,bm,color}
\usepackage{algorithm}
\usepackage{algorithmic}

\begin{document}

\theoremstyle{definition}
\newtheorem{definition}{Definition}
\newtheorem{lemma}{Lemma}
\newtheorem{theorem}{Theorem}
\newtheorem{corollary}{Corollary}
\newtheorem{example}{Example}
\newtheorem{remark}{Remark}
\newcommand{\dmu}{{\mathrm{d}\mu}}
\newcommand{\dnu}{{\mathrm{d}\nu}}
\newcommand{\dx}{{\mathrm{d}x}}
\newcommand{\dbmx}{{\mathrm{d}\bm{x}}}

\setlength{\baselineskip}{20pt}

\title{Constructing Markov chains with given dependence and marginal stationary distributions}
\author{Tomonari Sei\footnote{Graduate School of Information Science and Technology, The University of Tokyo.}}
\date{July, 2024}
\maketitle

\begin{abstract}
 A method of constructing Markov chains on finite state spaces is provided.
 The chain is specified by three constraints:
 stationarity, dependence and marginal distributions.
The generalized Pythagorean theorem in information geometry plays a central role in the construction.
An algorithm for obtaining the desired Markov chain is described.
 Integer-valued autoregressive processes are considered for illustration.
\end{abstract}

\section{Introduction}

Markov chains are fundamental in time series analysis.
A statistical model of (higher-order) Markov chains is typically described by a parametric family of Markov kernels.
Examples include the mixture transition distribution model \cite{Raftery1985,BerchtoldRaftery2002} and
the variable length Markov model \cite{Buhlmann1999}.
In these cases, the stationary distribution is not directly specified.

By contrast, if a parametric model of marginal stationary distributions is given,
its estimation is relatively easy because the methods for independent data can be formally applied under ergodicity.
Therefore, it is natural to consider a statistical model of Markov chains with given marginal stationary distributions.

In this paper, we provide a method of constructing stationary Markov chains with specified dependence and marginal distributions.
Here, the dependence refers to a central part of the Markov kernel; see Section~\ref{section:main} for the precise definition.
Our construction is based on the exponential family of Markov chains \cite{Nagaoka2005,HW2016}, which is defined in the framework of information geometry \cite{AmariNagaoka}.
In particular, the generalized Pythagorean theorem established by \cite{CCC1987} plays a central role.

The paper is organized as follows.
In Section~\ref{section:preliminaries},
we recall the generalized Pythagorean theorem on Markov chains.
In Section~\ref{section:main}, a method of constructing Markov chains is proposed.
A numerical algorithm and illustrative examples are provided in Section~\ref{section:algorithm}.
Future directions are discussed in Section~\ref{section:discussion}.

\section{Generalized Pythagorean theorem for Markov chains} \label{section:preliminaries}

We recall the definition of exponential families of Markov chains according to \cite{Nagaoka2005,HW2016}
and the generalized Pythagorean theorem proved by \cite{CCC1987}.

Let $\mathcal{X}$ be a finite set.
Let $\mathbb{R}_+$ and $\mathbb{R}_{\geq 0}$ be the set of positive and non-negative numbers, respectively.
The set of all positive probability distributions on $\mathcal{X}$ is denoted by $\mathcal{P}_+(\mathcal{X})\subset \mathbb{R}_+^{\mathcal{X}}$,
where $\mathbb{R}_+^{\mathcal{X}}$ is the set of all functions from $\mathcal{X}$ to $\mathbb{R}_+$.
A (first-order) Markov kernel on $\mathcal{X}$ is a function $w:\mathcal{X}^2\to\mathbb{R}_{\geq 0}$ such that
$\sum_{y\in\mathcal{X}}w(y|x)=1$ for any $x\in\mathcal{X}$.
A distribution $p\in\mathcal{P}_+(\mathcal{X})$ is called a stationary distribution of $w$
if $\sum_{x\in\mathcal{X}} w(y|x)p(x)=p(y)$ for any $y\in\mathcal{X}$.

Suppose that, throughout this section, we have a subset $\mathcal{E}$ of $\mathcal{X}^2$ such that the directed graph $(\mathcal{X},\mathcal{E})$ is strongly connected.
This means that for any $(x,y)\in\mathcal{X}^2$ there exists a path in $(\mathcal{X},\mathcal{E})$ from $x$ to $y$.

\begin{example}\label{example:strongly-connected}
 Consider a four-element set $\mathcal{X}=\{00,01,10,11\}$ and define $\mathcal{E}$ by
 $(ij,kl)\in\mathcal{E}$ if and only if $j=k$.
 Then, $(\mathcal{X},\mathcal{E})$ is strongly connected.
 Indeed, for given $(ij,kl)\in\mathcal{X}^2$, we have a path $ij\to jk\to kl$.
\end{example}

Let $\mathcal{W}=\mathcal{W}(\mathcal{X},\mathcal{E})$
denote the set of all Markov kernels $w$ with the property $\{(x,y)\in\mathcal{X}^2\mid w(y|x)>0\}=\mathcal{E}$.
From the Perron--Frobenius theorem (e.g.\ \cite{Zhang}), each $w\in\mathcal{W}$ has a unique stationary distribution $p_w\in\mathcal{P}_+(\mathcal{X})$.
Denote the joint stationary distribution induced from $w\in\mathcal{W}$ by
\[
 p_w^{(n)}(x_1,\ldots,x_n) = p_w(x_1)w(x_2|x_1)\cdots w(x_n|x_{n-1})
\]
for $n\geq 1$ and $(x_1,\ldots,x_n)\in\mathcal{X}^n$. In particular, $p_w^{(1)}=p_w$.

\begin{lemma}[\cite{CCC1987,Nagaoka2005,HW2016}]
\label{lemma:Nagaoka-thm1}
 Let $f:\mathcal{X}^2\to\mathbb{R}_{\geq 0}$ be given and suppose that $\{(x,y)\mid f(x,y)>0\}=\mathcal{E}$.
 Then, there exist $\kappa\in\mathbb{R}^\mathcal{X}$ and $\psi\in\mathbb{R}$ such that the function $w$ defined by
 \begin{align}
  w(y|x) = f(x,y)\exp(\kappa(y)-\kappa(x)-\psi)
  \label{eq:w-rep}
 \end{align}
 is a Markov kernel.
 Here, $\psi$ is unique and $\kappa$ is unique up to an additive constant.
\end{lemma}

\begin{proof}
 From the Perron--Frobenius theorem, there exist a unique positive eigenvalue $Z$ and
 its eigenvector $\gamma\in\mathbb{R}_+^\mathcal{X}$ such that $\sum_y f(x,y)\gamma(y)=Z\gamma(x)$.
 Then, $w(y|x)=f(x,y)\gamma(y)/(Z\gamma(x))$ is a Markov kernel.
Let $\psi=\log Z$ and $\kappa=\log\gamma$ to obtain (\ref{eq:w-rep}).
\end{proof}

 \begin{remark}[\cite{CCC1987,HW2016}]
 \label{remark:stationary}
 The stationary distribution of $w$ in (\ref{eq:w-rep}) is
 \[
 p_w(x)=\frac{\beta(x)\gamma(x)}{\sum_{x'}\beta(x')\gamma(x')},
 \]
 where $\beta$ and $\gamma$ are the left and right Perron--Frobenius eigenvectors of $f$.
 This fact is sometimes useful; see Section~\ref{section:algorithm}.
\end{remark}

Based on the lemma, we define exponential families of Markov kernels as follows.

\begin{definition}[exponential family of Markov kernels \cite{Nagaoka2005,HW2016}]
 Let $C,F_1,\ldots,F_K:\mathcal{E}\to\mathbb{R}$ be given.
 A family $\{w_\theta\mid\theta\in\mathbb{R}^K\}\subset\mathcal{W}$ is called the {\em exponential family} generated by $C,F_1,\ldots,F_K$
 if it is written as
 \begin{align}
  w_\theta(y|x) = \exp\left(C(x,y)+ \sum_{k=1}^K \theta_kF_k(x,y)+\kappa_\theta(y)-\kappa_\theta(x)-\psi_\theta\right)
  \label{eq:exp-family}
 \end{align}
 for $(x,y)\in\mathcal{E}$, where $\kappa_\theta\in\mathbb{R}^\mathcal{X}$ and $\psi_\theta\in\mathbb{R}$ are determined by Lemma~\ref{lemma:Nagaoka-thm1}.
 The parameter $\theta$ is called the {\em natural parameter} and $\psi_\theta$ is called the {\em potential}.
\end{definition}

\begin{remark}
Another type of exponential families of stochastic processes is discussed in \cite{KSbook},
where stationarity or the Markov property are not necessary in general.
One of the reasons why (\ref{eq:exp-family}) is called an exponential family
is that it has a dually flat structure \cite{Nagaoka2005}.
See Chapter~3 of \cite{AmariNagaoka} for details about dually flat spaces.
Recently, \cite{Nakajima2023} established an extended space of Markov kernels that has the same structure.
\end{remark}

We summarize fundamental properties of the exponential family.
Let $\mathcal{N}\subset\mathbb{R}^\mathcal{E}$ be the set of all functions of the form $f(x,y)=\kappa(y)-\kappa(x)-c$ for some $\kappa\in\mathbb{R}^\mathcal{X}$ and $c\in\mathbb{R}$.

\begin{lemma}[\cite{Nagaoka2005,HW2016}] \label{lemma:exp-family-properties}
 Let $\{w_\theta\mid\theta\in\mathbb{R}^K\}$ be the exponential family generated by $C,F_1,\ldots,F_K$.
 Suppose that $F_1,\ldots,F_K$ are linearly independent modulo $\mathcal{N}$.
 Then, the map $\theta\mapsto w_\theta$ is injective.
The potential $\psi_\theta$ is a strictly convex function of $\theta$.
The derivative $\partial_k\psi_\theta$  is equal to the expectation of $F_k(x,y)$ with respect to the stationary distribution,
where $\partial_k=\partial/\partial\theta_k$.
\end{lemma}

\begin{proof}
 The injectivity follows from the linear independence assumption.
 By taking derivatives of the identity $\sum_y w_\theta(y|x)=1$ with respect to $\theta$
 and then taking expectation with respect to $p_{w_\theta}$, we obtain the gradient and Hessian of $\psi_\theta$ as
 \[
 \partial_k\psi_\theta = \sum_{(x,y)\in\mathcal{E}} p_{w_\theta}^{(2)}(x,y)F_k(x,y)
 \]
 and
 \begin{align}
 \partial_k\partial_l\psi_\theta = \sum_{(x,y)\in\mathcal{E}} p_{w_\theta}^{(2)}(x,y)(\partial_k\log w_\theta(y|x))(\partial_l\log w_\theta(y|x)).
 \label{eq:Fisher}
 \end{align}
 The functions $\partial_k\log w_\theta(y|x)=F_k(x,y)+\partial_k\kappa_\theta(y)-\partial_k\kappa_\theta(x)-\partial_k\psi_\theta$
 for $1\leq k\leq K$ are linearly independent in $\mathbb{R}^{\mathcal{E}}$ because $F_1,\ldots,F_K$ are linearly independent modulo $\mathcal{N}$.
 Thus, the Hessian $\partial_k\partial_l\psi_\theta$ is positive definite and $\psi_\theta$ is strictly convex.
\end{proof}

\begin{remark}
 The right hand side of (\ref{eq:Fisher}) is called the Fisher information matrix,
 which is defined for any families not limited to exponential families.
\end{remark}

Define the divergence rate of Markov chains by
\[
 D(v|w) = \sum_{(x,y)\in\mathcal{E}}p_v^{(2)}(x,y)\log\frac{v(y|x)}{w(y|x)},\quad v,w\in\mathcal{W},
\]
which is nonnegative and becomes zero if and only if $v=w$.
It is not difficult to see that
the divergence rate is the limit of the normalized Kullback--Leibler divergence:
\[
 \lim_{n\to\infty}\frac{1}{n}\sum_{x_1,\ldots,x_n}p_v^{(n)}(x_1,\ldots,x_n)\log\frac{p_v^{(n)}(x_1,\ldots,x_n)}{p_w^{(n)}(x_1,\ldots,x_n)} = D(v|w).
\]

\begin{lemma}[Generalized Pythagorean theorem for Markov chains; Lemma~1 of \cite{CCC1987}] \label{lemma:Pythagorean}
 Let $E=\{w_\theta\mid\theta\in\mathbb{R}^K\}$ be the exponential family generated by $C,F_1,\ldots,F_K$,
 where $F_1,\ldots,F_K$ are linearly independent modulo $\mathcal{N}$.
 For given $\mu_1,\ldots,\mu_K\in\mathbb{R}$,
 let $M=M(\mu_1,\ldots,\mu_K)\subset\mathcal{W}$ be the set of all $w\in\mathcal{W}$ satisfying
 \begin{align}
   \sum_{(x,y)\in\mathcal{E}} p_w^{(2)}(x,y)F_k(x,y) = \mu_k,\ \ 1\leq k\leq K.
  \label{eq:mixture}
 \end{align}
 If $M\neq\emptyset$, then there exists a unique $w_*\in M\cap E$.
 Furthermore, the relation
 \begin{align}
 D(w|w_*) + D(w_*|v) = D(w|v)
 \label{eq:Pythagorean}
 \end{align}
 holds for any $w\in M$ and $v\in E$.
\end{lemma}

The most challenging part of the lemma is existence of $w_*$.
The lemma is a particular case of Lemma~1 of \cite{CCC1987}.
We give a proof in Appendix for ease of reference.

There are a couple of consequences of (\ref{eq:Pythagorean}).
Let $E$ and $M$ be defined as in Lemma~\ref{lemma:Pythagorean}.
If $v\in E$ is given a priori,
the minimization problem
\begin{align}
 \mathop{\rm Minimize}_{w\in M}\ \ D(w|v)
 \label{eq:e-projection}
\end{align}
has a unique solution $w_*\in M\cap E$.
The solution $w_*$ is called the {\em Markov $I$-projection} of $v$ onto $M$ in \cite{CCC1987}.
Similarly, if $w\in M$ is given a priori, the minimization problem
\begin{align}
 \mathop{\rm Minimize}_{v\in E}\ \ D(w|v)
 \label{eq:m-projection}
\end{align}
has a unique solution $w_*\in M\cap E$.
The projections (\ref{eq:e-projection}) and (\ref{eq:m-projection}) are also called 
e-projection and m-projection, respectively. See Section~3.5 of \cite{AmariNagaoka}.

The problem (\ref{eq:m-projection}) is written in terms of the natural parameter $\theta$ as follows.

\begin{corollary} \label{corollary:m-projection}
 Under the same notation as Lemma~\ref{lemma:Pythagorean}, the natural parameter of $w_*\in M\cap E$ is the solution of
 \[
  \mathop{\rm Minimize}_{\theta\in\mathbb{R}^K}
  \ \ \psi_\theta - \sum_{k=1}^K\theta_k \mu_k.
 \]
\end{corollary}

\begin{proof}
 For $w\in M$ and $v=w_\theta\in E$, we obtain
 \begin{align*}
  D(w|w_\theta)
  &= \sum_{(x,y)\in\mathcal{E}}p_w^{(2)}(x,y)
  \log\frac{w(y|x)}{e^{C(x,y)+\sum_{k=1}^K\theta_kF_k(x,y)+\kappa_\theta(y)-\kappa_\theta(x)-\psi_\theta}}
  \\
  &= A - \sum_{k=1}^K \theta_k \mu_k+\psi_\theta
 \end{align*}
 where $A$ is a term not depending on $\theta$.
\end{proof}


\section{Main results} \label{section:main}

\subsection{First-order Markov chains}

We begin with first-order Markov chains.
Assume $\mathcal{E}=\mathcal{X}^2$ throughout this subsection.
In particular, $\mathcal{W}=\mathcal{W}(\mathcal{X},\mathcal{X}^2)$ is the set of all strictly positive Markov kernels.

\begin{theorem} \label{theorem:1st}
Let $H:\mathcal{X}^2\to\mathbb{R}$ and $r\in\mathcal{P}_+(\mathcal{X})$ be given.
Then, there exist functions $\kappa\in\mathbb{R}^\mathcal{X}$ and $\delta\in\mathbb{R}^\mathcal{X}$ such that
a function
\begin{align}
 w(y|x) = \exp(H(x,y)+\kappa(y)-\kappa(x)-\delta(y)),\quad (x,y)\in \mathcal{X}^2,
 \label{eq:min-info}
\end{align}
is a Markov kernel and its stationary distribution is
\begin{align}
p_w(x)=r(x),\quad x\in\mathcal{X}.
\label{eq:marginal}
\end{align}
Here, $\delta$ is unique and $\kappa$ is unique up to an additive constant.
\end{theorem}

\begin{proof}
 We label the elements of $\mathcal{X}$ as $\mathcal{X}=\{\xi_1,\ldots,\xi_m\}$, where $m=|\mathcal{X}|$.
 Consider an exponential family $E$ generated by $C(x,y)=H(x,y)$
 and $F_i(x,y)=-I_{\{\xi_i\}}(y)$ for $1\leq i\leq m-1$, where $I_{\{\xi_i\}}(y)=1$ if $y=\xi_i$ and $0$ otherwise.
 Each element of $E$ is written as
 \[
  w_\theta(y|x) = \exp\left(H(x,y) - \sum_{i=1}^{m-1}\theta_i I_{\{\xi_i\}}(y) + \kappa_\theta(y)-\kappa_\theta(x)-\psi(\theta)\right),
 \]
 which is of the form (\ref{eq:min-info}) if we set $\kappa(y)=\kappa_\theta(y)$ and $\delta(y)=\sum_{i=1}^{m-1}\theta_iI_{\xi_i}(y)+\psi_\theta$.
 On the other hand, let $M$ be the set of all $w\in\mathcal{W}$ satisfying
 \[
  \sum_{(x,y)\in\mathcal{X}^2} p_w^{(2)}(x,y)F_i(x,y) = -r(\xi_i),\quad i=1,\ldots,m-1.
 \]
 This condition is equivalent to (\ref{eq:marginal}).
 Since a trivial Markov kernel $w(y|x)=r(y)$ obviously belongs to $M$, we have $M\neq\emptyset$.
 By Lemma~\ref{lemma:Pythagorean}, there exists a unique $w\in M\cap E$, which satisfies (\ref{eq:min-info}) and (\ref{eq:marginal}).
\end{proof}

\begin{definition}[minimum information Markov kernel] \label{definition:min-info}
 Let $H:\mathcal{X}^2\to\mathbb{R}$ and $r\in\mathcal{P}_+(\mathcal{X})$ be given.
 The Markov kernel $w$ determined by (\ref{eq:min-info}) and (\ref{eq:marginal}) is called 
 a {\em minimum information Markov kernel} generated by $H$ and $r$.
 The function $H$ is referred to as {\em dependence}.
\end{definition}

The term ``minimum information'' comes from the minimization problems (\ref{eq:e-projection}) and (\ref{eq:m-projection}); see \cite{BW2014,SY2023} for details.
Note that the dependence $H$ is not unique even if the Markov kernel $w$ is specified.
Indeed, $H(x,y)$ and $H(x,y)+\kappa_0(y)-\kappa_0(x)-\delta_0(y)$ for any $\kappa_0$ and $\delta_0$ induce the same $w$.


\begin{example}[An integer-valued autoregressive model] \label{example:INAR(1)}
 Consider a state space $\mathcal{X}=\{0,1,\cdots,N\}$ with $N\geq 1$.
 Let $H(x,y)=\alpha xy$ and $r(x)=\binom{N}{x}\nu^x(1-\nu)^{N-x}$, where $\alpha\in\mathbb{R}$ and $\nu\in(0,1)$ are parameters.
 The minimum information Markov kernel generated by $H$ and $r$
 is a kind of autoregressive models with binomial marginals.
 Numerical results are provided in Section~\ref{section:algorithm}.
\end{example}

\begin{remark} \label{remark:Sinkhorn}
Theorem~\ref{theorem:1st} is closely related to a classical matrix scaling problem \cite{SinkhornKnopp1967}.
The problem is to find $\beta\in\mathbb{R}_+^\mathcal{X}$ and $\gamma\in\mathbb{R}_+^\mathcal{X}$ such that
 \[
 \sum_y e^{H(x,y)}\beta(x)\gamma(y) = r(x),
 \quad \sum_x e^{H(x,y)}\beta(x)\gamma(y) = r(y),
 \]
 for given $H$ and $r$.
 We have the solution $\beta(x)=e^{-\kappa(x)}$ and $\gamma(y)=e^{\kappa(y)-\delta(y)}$ by using $\kappa$ and $\delta$ in Theorem~\ref{theorem:1st}.
 In other words, Theorem~\ref{theorem:1st} is just a corollary of the existing result.
 However, this correspondence no longer holds for higher-order Markov chains.
\end{remark}

\subsection{Higher-order Markov chains} \label{subsection:higher}

We now consider higher-order Markov chains.
Let $d\geq 1$.
A sequence $(x_1,\ldots,x_d)\in\mathcal{X}^d$ is abbreviated as $x_{1:d}$.
Define $x_{s:t}$ for $s\leq t$ as well.
A $d$-th-order Markov kernel is
a function $w:\mathcal{X}^{d+1}\to\mathbb{R}_{\geq 0}$ such that $\sum_{y\in\mathcal{X}}w(y|x_{1:d})=1$
for any $x_{1:d}\in\mathcal{X}^d$.
Denote the set of strictly positive $d$-th-order Markov kernels by $\mathcal{W}_d$.
Each $w\in\mathcal{W}_d$ is identified with a first-order Markov kernel $\tilde{w}$ on the graph $(\mathcal{X}^d,\mathcal{E})$ with an edge set
\begin{align}
\mathcal{E}=\{(x_{1:d},x_{2:(d+1)})\in\mathcal{X}^d\times\mathcal{X}^d\mid 
x_1,\ldots,x_{d+1}\in\mathcal{X}\}.
\label{eq:lift}
\end{align}
More specifically, define $\tilde{w}\in\mathcal{W}(\mathcal{X}^d,\mathcal{E})$ by
\begin{align}
 \tilde{w}(x_{2:(d+1)}|x_{1:d})
 = w(x_{d+1}|x_{1:d}).
\label{eq:identification}
\end{align}
It is proved in the same manner as Example~\ref{example:strongly-connected}
that $(\mathcal{X}^d,\mathcal{E})$ is strongly connected.
The stationary distribution of $w\in\mathcal{W}_d$ is denoted as $p_w^{(d)}(x_{1:d})$, that is,
\[
 \sum_{x_1} p_w^{(d)}(x_{1:d})w(x_{d+1}|x_{1:d}) = p_w^{(d)}(x_{2:(d+1)}).
\]
The first-order stationary distribution $p_w^{(1)}(x_1)$ is well defined by marginalization $\sum_{x_{2:d}}p_w^{(d)}(x_{1:d})$.

\begin{theorem} \label{theorem:2nd}
 Let $H:\mathcal{X}^{d+1}\to\mathbb{R}$ and $r\in\mathcal{P}_+(\mathcal{X})$ be given.
 Then, there exist functions $\kappa\in\mathbb{R}^{\mathcal{X}^d}$ and $\delta\in\mathbb{R}^\mathcal{X}$ such that
 a function
 \begin{align}
  w(x_{d+1}|x_{1:d}) = \exp\left(H(x_{1:(d+1)}) + \kappa(x_{2:(d+1)}) - \kappa(x_{1:d}) - \delta(x_{d+1})\right)
  \label{eq:min-info-2}
 \end{align}
 belongs to $\mathcal{W}_d$
 and its first order stationary distribution is
 \begin{align}
  p_w^{(1)}(x_1) = r(x_1).
  \label{eq:marginal-2}
 \end{align}
\end{theorem}

\begin{proof}
 Let $\mathcal{E}$ be the edge set defined by (\ref{eq:lift}).
 As in the proof of Theorem~\ref{theorem:1st}, we denote $\mathcal{X}=\{\xi_1,\ldots,\xi_m\}$.
 Consider an exponential family $E\subset\mathcal{W}(\mathcal{X}^d,\mathcal{E})$ generated by $C(x_{1:d},x_{2:(d+1)})=H(x_{1:(d+1)})$
 and $F_i(x_{1:d},x_{2:(d+1)})=-I_{\{\xi_i\}}(x_{d+1})$ for $1\leq i\leq m-1$.
 The set $E$ coincides with the set of all functions $w\in\mathcal{W}_d$ of the form (\ref{eq:min-info-2}) under the identification (\ref{eq:identification}).
 Let $M$ be the set of all $\tilde{w}\in\mathcal{W}(\mathcal{X}^d,\mathcal{E})$ satisfying
 \[
 \sum_{(\tilde{x},\tilde{y})\in\mathcal{E}} p_{\tilde{w}}^{(2)}(\tilde{x},\tilde{y}) F_i(\tilde{x},\tilde{y}) = -r(\xi_i),
 \]
 or equivalently,
 \[
 \sum_{x_{1:(d+1)}\in\mathcal{X}^{d+1}}p_w^{(d+1)}(x_{1:(d+1)})I_{\{\xi_i\}}(x_{d+1})=r(\xi_i).
 \]
 This condition is equivalent to (\ref{eq:marginal-2}).
 Since a trivial Markov kernel $r(x_{d+1})$ obviously belongs to $M$, we have $M\neq\emptyset$.
 By Lemma~\ref{lemma:Pythagorean}, there exists a unique $\tilde{w}\in M\cap E$.
 The corresponding $w\in\mathcal{W}_d$ satisfies (\ref{eq:min-info-2}) and (\ref{eq:marginal-2}).
\end{proof}

\begin{definition}[higher-order minimum information Markov chain] \label{definition:min-info-2}
 Let $d\geq 1$, $H:\mathcal{X}^{d+1}\to\mathbb{R}$ and $r\in\mathcal{P}_+(\mathcal{X})$ be given.
 The Markov kernel $w$ determined by (\ref{eq:min-info-2}) and (\ref{eq:marginal-2}) is called 
 a {\em $d$-th order minimum information Markov kernel} generated by $H$ and $r$.
 The function $H$ is referred to as {\em dependence}.
\end{definition}

We provide an example of integer-valued autoregressive processes of higher orders in Section~\ref{section:algorithm}.

\section{Numerical examples} \label{section:algorithm}

We describe an algorithm to find the Markov $I$-projection $w_*$ using Corollary~\ref{corollary:m-projection}.
The minimum information Markov kernels are then numerically obtained.

\subsection{Computation of the Markov $I$-projection}

Consider the exponential family generated by $C,F_1,\ldots,F_K$ and the mixture family $M=M(\mu_1,\ldots,\mu_K)$ as in Lemma~\ref{lemma:Pythagorean}.
The unique intersection $w_*\in M\cap E$ is obtained by the minimization problem
\begin{align}
 \mathop{\rm Minimize}_{\theta\in\mathbb{R}^K}\ \ \psi_\theta - \sum_{k=1}^K \theta_k\mu_k
 \label{eq:algo-1-problem}
\end{align}
as shown in Corollary~\ref{corollary:m-projection}.
The problem is numerically solved by gradient descent algorithms such as the Broyden--Fletcher--Goldfarb--Shanno  (BFGS) algorithm.
The gradient of $\psi_\theta$ is
\begin{align}
 \partial_k\psi_\theta = \sum_{(x,y)\in\mathcal{E}}p_{w_\theta}^{(2)}(x,y)F_k(x,y)
 \label{eq:algo-1-mean}
\end{align}
from Lemma~\ref{lemma:exp-family-properties}.
We summarize the overall algorithm in Algorithm~\ref{algorithm:I-projection}.

\begin{algorithm}[t]
\caption{The Markov $I$-projection}
\label{algorithm:I-projection}
\begin{algorithmic}[1]
\REQUIRE $(\mathcal{X},\mathcal{E})$, $C,F_1,\ldots,F_K\in\mathbb{R}^{\mathcal{E}}$, $\mu_1,\ldots,\mu_K\in\mathbb{R}$ and tolerance $\varepsilon>0$
\ENSURE $\theta$ that solves (\ref{eq:algo-1-problem}), together with $\kappa_\theta,\psi_\theta,p_{w_\theta}$
\STATE $\theta=(0,\ldots,0)\in\mathbb{R}^K$
\REPEAT
\STATE Find $\kappa_\theta$ and $\psi_\theta$ in (\ref{eq:exp-family}) by the Perron--Frobenius theorem
\STATE Find the stationary distribution $p_{w_\theta}$ (see Remark~\ref{remark:stationary})
\STATE Compute $f=\psi_\theta-\sum_k \theta_k\mu_k$
\STATE Compute $g_k=\partial\psi_\theta/\partial\theta_k - \mu_k$ for $1\leq k\leq K$ by using (\ref{eq:algo-1-mean})
\STATE Update $\theta$ using $f$ and $g$ (e.g.\ by the BFGS method)
\UNTIL The reduction of $f$ is within $\varepsilon$
\RETURN $\theta$, $\kappa_\theta$, $\psi_\theta$, $p_{w_\theta}$
\end{algorithmic}
\end{algorithm}

In \cite{HW2016}, parameter estimation of $\theta$ for given observation $\{x_t\}_{t=1}^n$ is considered
and an algorithm based on (\ref{eq:algo-1-problem}) is proposed,
where $\mu_k=\bar{F}_k=(n-1)^{-1}\sum_{t=2}^nF_k(x_{t-1},x_t)$ is the sample mean of the statistic $F_k$.

\subsection{The first-order case}

The first-order minimum information Markov kernel is computed as follows.
Our goal is to find the functions $\kappa$ and $\delta$ in (\ref{eq:min-info})
for given $H$ and $r$.

Let $\mathcal{X}=\{\xi_1,\ldots,\xi_m\}$, $\mathcal{E}=\mathcal{X}^2$,
$C(x,y)=H(x,y)$, $F_i(x,y)=-I_{\{\xi_i\}}(y)$ and $\mu_i=-r(\xi_i)$ for $1\leq i\leq m-1$
as considered in the proof of Theorem~\ref{theorem:1st}.
Then, apply Algorithm~\ref{algorithm:I-projection} to obtain the optimal $\theta=(\theta_1,\ldots,\theta_{m-1})$ together with $\kappa_\theta$ and $\psi_\theta$.
Finally, let
\[
 \kappa(y)=\kappa_\theta(y)
 \quad \mbox{and}\quad 
 \delta(y)=\sum_{i=1}^{m-1}\theta_iI_{\{\xi_i\}}(y)+\psi_\theta.
\]
%

\begin{example}[Continuation of Example~\ref{example:INAR(1)}] \label{example:INAR(1)-cont}
 Let $\mathcal{X}=\{0,1,\cdots,N\}$, $H(x,y)=\alpha xy$ and $r(x)=\binom{N}{x}\nu^x(1-\nu)^{N-x}$, as described in Example~\ref{example:INAR(1)}.
 Figure~\ref{fig:INAR(1)} shows a sample path $\{x_t\}_{t=1}^n$, autocorrelation function, partial autocorrelation function and marginal distribution
 of the minimum information Markov kernel $w$
 when $\alpha=-1$, $N=5$, $\nu=0.4$ and $n=365$.
 The sign of the autocorrelation of order 1 is negative.
 This is intuitively explained by the joint distribution of $x_1,\ldots,x_n$:
 \[
  p_w^{(n)}(x_{1:n}) = 
  r(x_1)\exp\left(
   \alpha\sum_{t=2}^n x_{t-1}x_t
   +\kappa(x_n)-\kappa(x_1)-\sum_{t=2}^n\delta(x_t)
  \right).
 \]
 A path with a small value of $\sum_{t=2}^nx_{t-1}x_t$ will be observed more likely since $\alpha<0$.
 The partial autocorrelation is almost zero for lag greater than 1, as expected from the Markov structure.
 \end{example}
 
 \begin{figure}[t]
 \centering
 \includegraphics[width=0.8\textwidth]{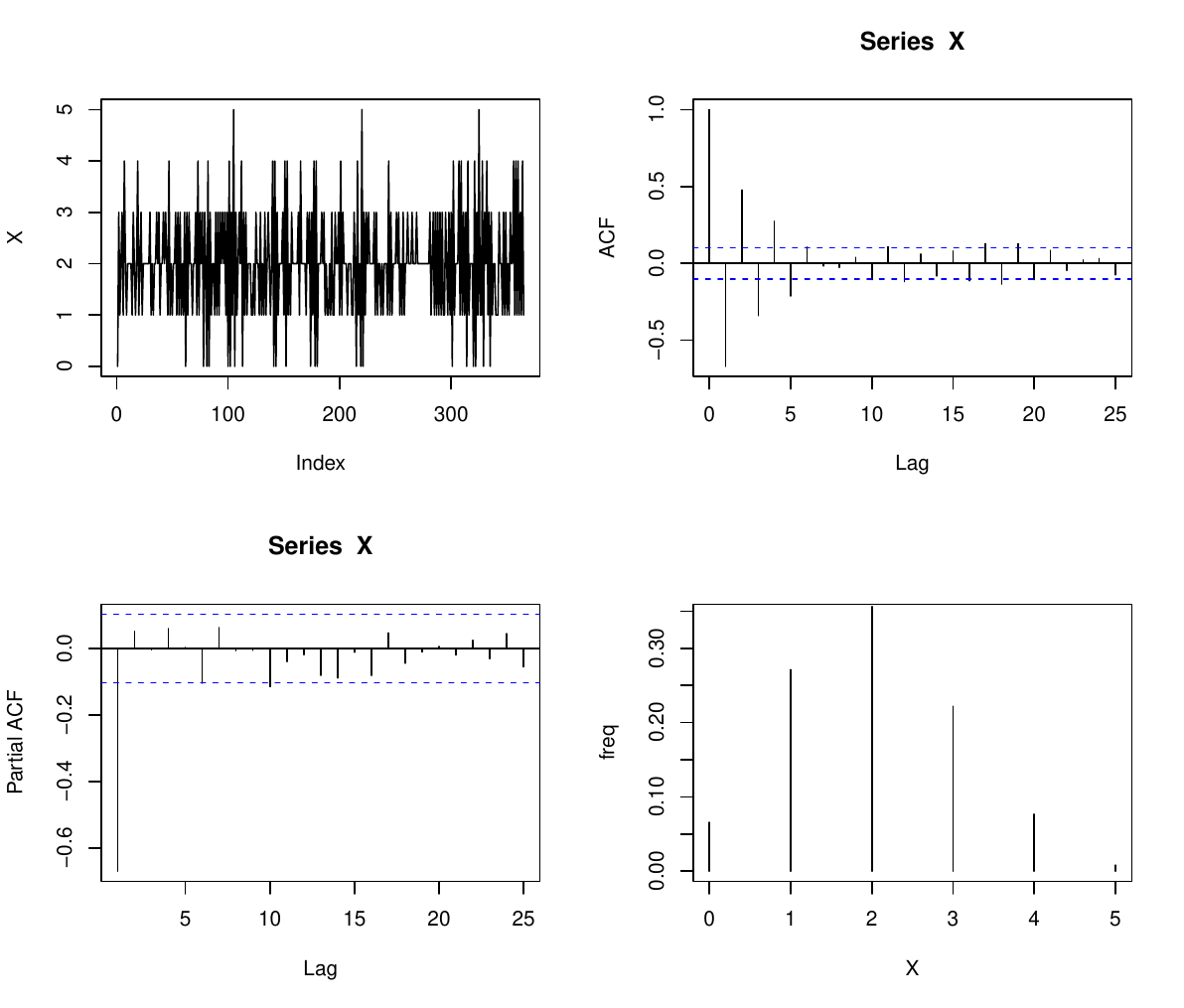}
 \caption{A sample path, autocorrelation function, partial autocorrelation function and marginal distribution of the integer-valued autoregressive process of order 1, where $\alpha=-1$, $N=5$, $\nu=0.4$ and $n=365$.}
 \label{fig:INAR(1)}
 \end{figure}

\subsection{Higher-order cases}

The $d$-th-order minimum information Markov kernels are similarly computed as follows.
The goal is to find the functions $\kappa$ and $\delta$ in (\ref{eq:min-info-2}) and the stationary distribution $p_w^{(d)}$
for given $H$ and $r$.

Let $\mathcal{X}=\{\xi_1,\ldots,\xi_m\}$ and define the edge set $\mathcal{E}\subset\mathcal{X}^d\times\mathcal{X}^d$ by (\ref{eq:lift}).
Let $C(x_{1:d},x_{2:(d+1)})=H(x_{1:(d+1)})$, $F_i(x_{1:d},x_{2:(d+1)})=-I_{\{\xi_i\}}(x_{d+1})$ and $\mu_i=-r(\xi_i)$ for $1\leq i\leq m-1$
as considered in the proof of Theorem~\ref{theorem:2nd}.
Then, apply Algorithm~\ref{algorithm:I-projection} to obtain the optimal $\theta=(\theta_1,\ldots,\theta_{m-1})$ together with $\kappa_\theta$, $\psi_\theta$ and $p_{\tilde{w}_\theta}$,
where $\tilde{w}_\theta$ denotes the element of $\mathcal{W}(\mathcal{X}^d,\mathcal{E})$.
Finally, let
\[
 \kappa(x_{1:d})=\kappa_\theta(x_{1:d}),
 \quad
 \delta(x_{d+1})=\sum_{i=1}^{m-1}\theta_iI_{\{\xi_i\}}(x_{d+1})+\psi_\theta
 \quad \mbox{and}\quad p_w^{(d)}=p_{\tilde{w}_\theta}.
\]


\begin{example}[An integer-valued autoregressive model of order 2]
 Consider $\mathcal{X}=\{0,1,\cdots,N\}$ with a positive integer $N$.
 Let $H(x,y,z)=\alpha_1 yz + \alpha_2 xz$.
 Let $r$ be the binomial distribution as in the preceding example.
 The minimum information Markov kernel provides a second-order stationary autoregressive model with binomial marginals.
 Figure~\ref{fig:INAR(2)} shows a sample path, autocorrelation function, partial autocorrelation function and marginal distribution
 of the minimum information Markov kernel $w$
 when $\alpha=(0.6, -0.3)$, $N=5$, $\nu=0.4$ and $n=365$.
 The sign of $\alpha_2$ coincides with that of the partial autocorrelation of order~2.
 This result is expected from the construction.
\end{example}

 \begin{figure}[t]
 \centering
 \includegraphics[width=0.8\textwidth]{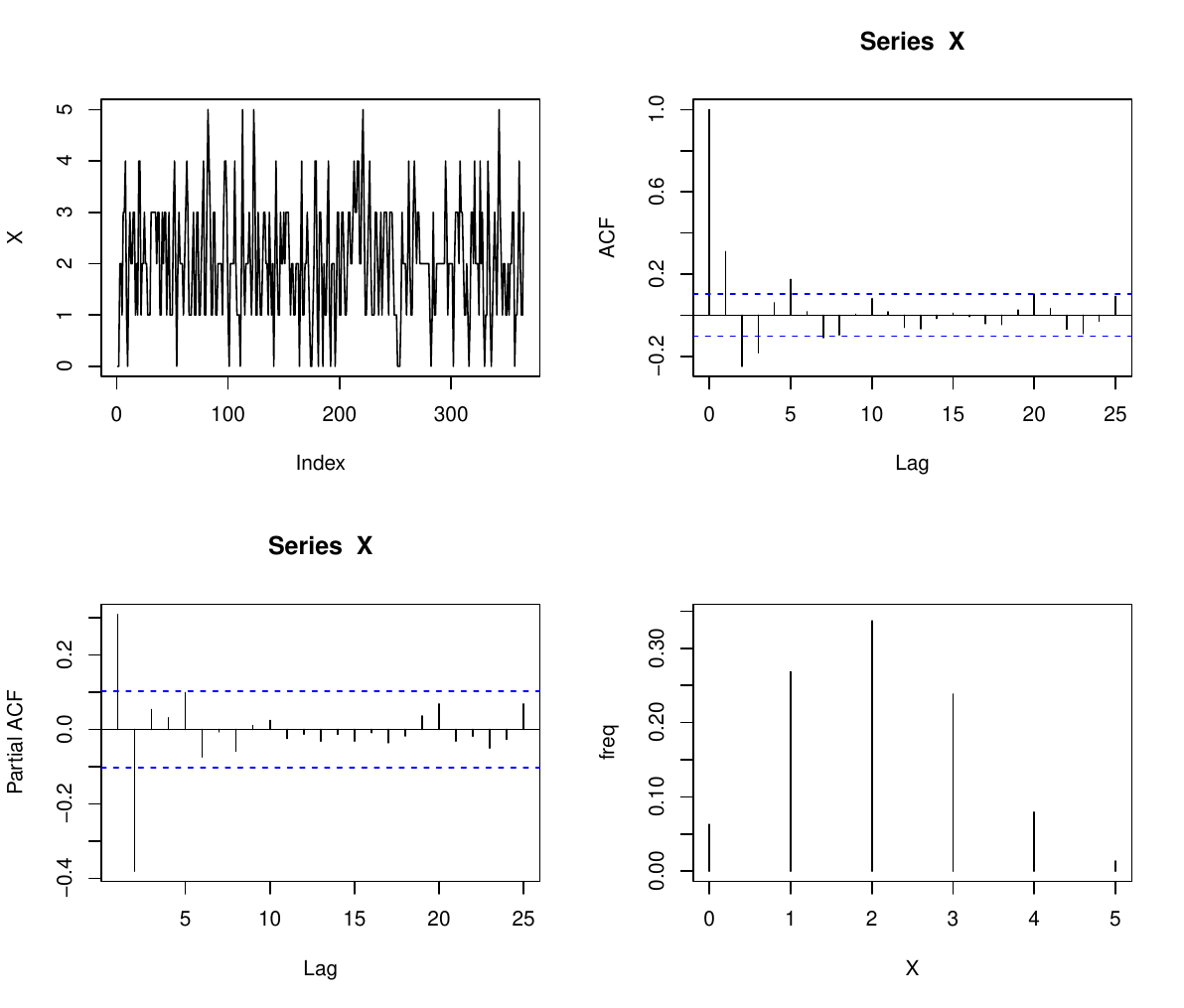}
 \caption{A sample path, autocorrelation function, partial autocorrelation function and marginal distribution of the integer-valued autoregressive process of order 2, where $\alpha=(0.6,-0.3)$, $N=5$, $\nu=0.4$ and $n=365$.}
 \label{fig:INAR(2)}
 \end{figure}

\section{Discussion} \label{section:discussion}

\subsection{Infinite state spaces}

It is conjectured that the above construction of Markov models will be valid even for infinite state spaces under mild regularity conditions.
Technical difficulty comes from the existence part of Lemma~\ref{lemma:Pythagorean}.
It is known that the corresponding theorem for independent sequences holds under fairly general conditions \cite{Csiszar1975,Nutz2022}.

Stationary Gaussian autoregressive models are understood in our framework without the technical difficulty.
Details will be given in a forthcoming paper.

\subsection{Statistical inference}

In this paper, we focused on construction of a Markov chain generated by $H$ and $r$.
In statistics, some parametric models for $H$ and $r$ are assumed
and the parameters are estimated from a given time series data.
Estimators and their sampling properties have to be investigated.

For this problem, we can use the Markov $I$-projection in a similar manner to \cite{HW2016}.
Consider a parametric model of first-order Markov chains on $\mathcal{X}=\{\xi_1,\ldots,\xi_m\}$
specified by $H(x,y)=h_0(x,y)+\sum_{k=1}^K\theta_k h_k(x,y)$, where $\theta$ is the unknown parameter.
Then, the Markov kernel
\[
\exp\left(h_0(x,y)+\sum_{k=1}^K\theta_k h_k(x,y)+\kappa(y)-\kappa(x)-\delta(y)\right)
\]
is considered as an exponential family generated by $C=h_0,F_k=h_k$ for $1\leq k\leq K$
and $F_{K+i}=-I_{\{\xi_i\}}$ for $1\leq i\leq m-1$.
The parameters $\theta$ and $\delta$ are simultaneously estimated by Algorithm~\ref{algorithm:I-projection},
where $\mu_k=\bar{h}_k$ and $\mu_{K+i}=-\bar{I}_{\{\xi_i\}}$ are sample means of the corresponding statistics.

However, for higher-order Markov chains, the curse of dimensionality occurs.
In the minimum information dependence modeling \cite{SY2023}, which is intended for i.i.d.\ data,
a conditional likelihood estimator is considered.
It is a future work to construct a similar estimating method for Markov models.

\section*{Acknowledgements}
The author thanks Keisuke Yano, Kentaro Tanaka and Issey Sukeda for their helpful comments.
This work is supported by 	JSPS KAKENHI 19K11865 and 21K11781.

\newpage
\appendix
\section*{Appendix}
\section{Proof of Lemma~\ref{lemma:Pythagorean}}

 We first prove that the relation (\ref{eq:Pythagorean}) holds if there exists $w_*\in M\cap E$.
 Denote the natural parameter of $v$ and $w_*$ by $\theta$ and $\theta_*$, respectively.
 Then, we have
 \begin{align*}
 & D(w|w_*)+D(w_*|v)-D(w|v)
 \\
 &= \sum_{(x,y)\in\mathcal{E}} (p_w^{(2)}(x,y) - p_{w_*}^{(2)}(x,y))\log\frac{v(y|x)}{w_*(y|x)}
 \\
 &= \sum_{(x,y)\in\mathcal{E}} (p_w^{(2)}(x,y) - p_{w_*}^{(2)}(x,y))
 \\
 &\quad\quad
 \times\left(
 \sum_k(\theta_k-\theta_{*k}) F_k(x,y) + \kappa_\theta(y)-\kappa_\theta(x) -\psi_\theta
 - \kappa_{\theta_*}(y) + \kappa_{\theta_*}(x) + \psi_{\theta*}
 \right)
 \\
 &= 0,
 \end{align*}
 where the last equality follows from the definition of $M$ and stationarity.
 
 Uniqueness of $w_*$ follows from (\ref{eq:Pythagorean}).
 Indeed, if $w\in M\cap E$,
 we can set $v=w$ in (\ref{eq:Pythagorean}) so that $D(w|w_*)+D(w_*|w)=0$,
 which implies $w=w_*$.

 For the proof of existence, we define some notations.
 Let $\mathcal{P}_{\rm s}$ be the set of all stationary joint distributions $p$ on $\mathcal{X}^2$ supported on $\mathcal{E}$,
 which means $\sum_y p(x,y)=\sum_y p(y,x)$ for any $x\in\mathcal{X}$ and $\{(x,y)\in\mathcal{X}^2\mid p(x,y)>0\}=\mathcal{E}$.
 Through Lemma~\ref{lemma:Nagaoka-thm1}, $\mathcal{P}_{\rm s}$ can be identified with $\mathcal{W}$.
 The marginal and conditional distributions of $p\in\mathcal{P}_{\rm s}$ are denoted by $\bar{p}(x)$ and $p(y|x)=p(x,y)/\bar{p}(x)$, respectively.
 The closure and relative boundary of $\mathcal{P}_{\rm s}$ as a subset of $\mathbb{R}^{\mathcal{X}^2}$ are denoted as ${\rm cl}(\mathcal{P}_{\rm s})$ and $\partial\mathcal{P}_{\rm s}$, respectively.
 It follows that $p\in{\rm cl}(\mathcal{P}_{\rm s})$ (resp.\ $p\in\partial\mathcal{P}_{\rm s}$) if and only if the support of $p$ is a subset (resp.\ proper subset) of $\mathcal{E}$.
 
 Choose any $v_0\in E$ and consider a function
 \[
 G(p) = \sum_{(x,y)\in\mathcal{E}} p(x,y)\log\frac{p(y|x)}{v_0(y|x)}
 \]
 of $p\in\mathcal{P}_{\rm s}$, which is nothing but the divergence rate from $p(y|x)$ to $v_0(y|x)$.
 An equivalent form
 \[
 G(p) = \sum_{(x,y)\in\mathcal{E}} p(x,y)\log\frac{p(x,y)}{v_0(y|x)}
 - \sum_{x\in\mathcal{X}} \bar{p}(x)\log\bar{p}(x)
 \]
 is well defined for any $p\in{\rm cl}(\mathcal{P}_{\rm s})$, where $0\log 0=0$,
 and is continuous on ${\rm cl}(\mathcal{P}_{\rm s})$.

 Let us prove that $G$ is strictly convex.
 Take two points $p_0\neq p_1$ in $\mathcal{P}_{\rm s}$
 and let $p_t=(1-t)p_0+tp_1$ for $t\in[0,1]$.
 Then, we have
 \begin{align}
  \frac{d}{dt}G(p_t)
  = \sum_{(x,y)\in\mathcal{E}} (p_1(x,y)-p_0(x,y))\log\frac{p_t(y|x)}{v_0(y|x)}
  \label{eq:G-deriv}
 \end{align}
 and
 \begin{align*}
  \left.\frac{d^2}{dt^2}G(p_t)\right|_{t=0}
  &= \sum_{(x,y)\in\mathcal{E}} (p_1(x,y)-p_0(x,y))\left(\frac{p_1(x,y)-p_0(x,y)}{p_0(x,y)}-\frac{\bar{p}_1(x)-\bar{p}_0(x)}{\bar{p}_0(x)}\right)\\
  &= \sum_{(x,y)\in\mathcal{E}} p_0(x,y)\left(\frac{p_1(x,y)}{p_0(x,y)}-1\right)\left(\frac{p_1(x,y)}{p_0(x,y)}-\frac{\bar{p}_1(x)}{\bar{p}_0(x)}\right)\\
  &= \sum_{(x,y)\in\mathcal{E}} p_0(x,y)\left(\frac{p_1(x,y)}{p_0(x,y)}-\frac{\bar{p}_1(x)}{\bar{p}_0(x)}\right)^2\\
  &\quad\quad + \sum_x\left(\frac{\bar{p}_1(x)}{\bar{p}_0(x)}-1\right) \sum_y  p_0(x,y)\left(\frac{p_1(x,y)}{p_0(x,y)}-\frac{\bar{p}_1(x)}{\bar{p}_0(x)}\right)\\
  &= \sum_{(x,y)\in\mathcal{E}} p_0(x,y)\left(\frac{p_1(x,y)}{p_0(x,y)}-\frac{\bar{p}_1(x)}{\bar{p}_0(x)}\right)^2.
 \end{align*}
 The last quantity is strictly positive since $p_1\neq p_0$ in $\mathcal{P}_{\rm s}$
 implies $p_1(y|x)\neq p_0(y|x)$ for some $(x,y)\in\mathcal{E}$.
 This implies $G$ is strictly convex.
 
 We prove that $G$ is steep at the boundary of $\mathcal{P}_{\rm s}$, which means
 \begin{align}
 \lim_{t\to+0}\frac{d}{dt}G(p_t)=-\infty
 \label{eq:steep}
 \end{align}
 for any $p_0\in\partial\mathcal{P}_{\rm s}$, $p_1\in\mathcal{P}_{\rm s}$ and $p_t=(1-t)p_0+tp_1$.
 Let $\mathcal{E}_0\subsetneq\mathcal{E}$ be the support of $p_0$.
 We prove that $\lim_{t\to +0} p_t(y|x)=0$ for some $(x,y)\in\mathcal{E}\setminus\mathcal{E}_0$,
 which implies (\ref{eq:steep}) due to the expression (\ref{eq:G-deriv}).
 Note that (\ref{eq:G-deriv}) is valid even for $p_0\in\partial\mathcal{P}_{\rm s}$ if $t\in(0,1)$.
 We consider two cases:

 Case (i): Assume $\bar{p}_0(x)>0$ for any $x\in\mathcal{X}$.
 If $(x,y)\in\mathcal{E}\setminus\mathcal{E}_0$, then
 \[
 p_t(y|x)=\frac{p_t(x,y)}{\bar{p}_t(x)}=\frac{tp_1(x,y)}{(1-t)\bar{p}_0(x)+t\bar{p}_1(x)}\to +0
 \]
 as $t\to+0$.
 
 Case (ii): Assume $\bar{p}_0(x)=0$ for some $x\in\mathcal{X}$.
 Let $\mathcal{X}_0\subsetneq\mathcal{X}$ be the support of $\bar{p}_0$.
 Since $\mathcal{E}$ is strongly connected, there exists $(x,y)\in\mathcal{E}\cap(\mathcal{X}_0\times\mathcal{X}_0^{\rm c})$.
 For such $(x,y)$, we have $\bar{p}_0(x)>0$, $p_0(x,y)\leq \bar{p}_0(y)=0$ and $p_1(x,y)>0$ so that
 \[
  p_t(y|x)=\frac{p_t(x,y)}{\bar{p}_t(x)}=\frac{tp_1(x,y)}{(1-t)\bar{p}_0(x)+t\bar{p}_1(x)}\to +0.
 \]

 We have proved that $G$ is continuous on the compact set ${\rm cl}(\mathcal{P}_{\rm s})$, strictly convex and steep.
 Now, let $M_{\rm s}\subset\mathcal{P}_{\rm s}$ be the set of $p\in\mathcal{P}_{\rm s}$ satisfying
 \begin{align*}
 & \sum_{(x,y)\in\mathcal{E}} p(x,y) F_k(x,y) = \mu_k,\quad k=1,\ldots,K.
\end{align*}
 From the above properties of $G$,
 the minimization problem of $G$ over $M_{\rm s}$ has a unique solution $p_*$ whenever $M_{\rm s}\neq\emptyset$.
 Since $M_{\rm s}$ is relatively open,
 the derivative (\ref{eq:G-deriv}) at $t=0$ with $p_0=p_*$ is zero for any $p_1\in M_{\rm s}$.
 This implies
 \[
  \log\frac{p_*(y|x)}{v_0(y|x)} = \kappa(y) - \kappa(x) + \sum_{k=1}^K \theta_k F_k(x,y) - \psi
 \]
 for some $\kappa$, $\theta$ and $\psi$, which are nothing but the Lagrange multipliers. Hence, $p_*(y|x)$ belongs to $M\cap E$.
 This completes the proof.

\end{document}